\pdfoutput=1
\RequirePackage{ifpdf}
\ifpdf 
\documentclass[pdftex]{sigma}
\else
\documentclass{sigma}
\fi

\numberwithin{equation}{section}

\newtheorem{Theorem}{Theorem}[section]

\newtheorem{Lemma}[Theorem]{Lemma}
\newtheorem{Proposition}[Theorem]{Proposition}
 { \theoremstyle{definition}

\newtheorem{Remark}[Theorem]{Remark}
}

\def\F{\mathbb F}

\def\Q{\mathbb Q}

\def\Z{\mathbb Z}
\def\mod{\ \mathrm{mod}\ }

\def\gen#1{\langle #1\rangle}
\def\({\left(}
\def\){\right)}
\def\Fx{\mathbb F^{\times}}
\def\wFx{\widehat{\mathbb F_p^{\times}}}
\def\CC#1#2{\binom {#1}{#2}}
\def\JS#1#2{\left(\frac{#1}{#2}\right)}

\def\fp{\mathfrak p}
\def\fP{\mathfrak P}
\def\ol{\overline}

\def\eps{\varepsilon}

\def\hgq#1#2#3#4{
{}_2F_1\(\begin{matrix}
 #1& #2\\
 & #3
 \end{matrix};#4\)}

\begin{document}

\allowdisplaybreaks

\newcommand{\arXivNumber}{1711.05842}

\renewcommand{\thefootnote}{}

\renewcommand{\PaperNumber}{050}

\FirstPageHeading

\ShortArticleName{Evaluation of Certain Hypergeometric Functions over Finite Fields}

\ArticleName{Evaluation of Certain Hypergeometric Functions \\ over Finite Fields\footnote{This paper is a~contribution to the Special Issue on Modular Forms and String Theory in honor of Noriko Yui. The full collection is available at \href{http://www.emis.de/journals/SIGMA/modular-forms.html}{http://www.emis.de/journals/SIGMA/modular-forms.html}}}

\Author{Fang-Ting TU~$^\dag$ and Yifan YANG~$^\ddag$}

\AuthorNameForHeading{F-T.~Tu and Y.~Yang}

\Address{$^\dag$~Department of Mathematics, 303 Lockett Hall, Louisiana State University, \\
\hphantom{$^\dag$}~Baton Rouge, LA 70803, USA}
\EmailD{\href{mailto:ftu@lsu.edu}{ftu@lsu.edu}}

\Address{$^\ddag$~Department of Mathematics, National Taiwan University and National Center\\
\hphantom{$^\ddag$}~for Theoretical Sciences, Taipei, Taiwan 10617, ROC}
\EmailD{\href{mailto:yangyifan@ntu.edu.tw}{yangyifan@ntu.edu.tw}}

\ArticleDates{Received November 17, 2017, in final form May 09, 2018; Published online May 19, 2018}

\Abstract{For an odd prime $p$, let $\phi$ denote the quadratic character of the multiplicative group $\F_p^\times$, where $\F_p$ is the finite field of $p$ elements. In this paper, we will obtain evaluations of the
hypergeometric functions $ {}_2F_1\(\begin{matrix} \phi\psi& \psi\\ & \phi \end{matrix};x\)$, $x\in \F_p$, $x\neq 0, 1$, over $\F_p$ in terms of Hecke character attached to CM elliptic curves for characters $\psi$ of $\F_p^\times$ of order $3$, $4$, $6$, $8$, and $12$.}

\Keywords{hypergeometric functions over finite fields; character sums; Hecke characters}

\Classification{11T23; 11T24; 11G05; 11G30}

\renewcommand{\thefootnote}{\arabic{footnote}}
\setcounter{footnote}{0}

\section{Introduction}
Given a prime power $q=p^k$, let $\F:=\F_q$ denote the finite field of $q$ elements and $\widehat\Fx$ be the group of multiplicative characters of~$\Fx$. For a character~$\chi$ in $\widehat\Fx$, we extend $\chi$ to a function on~$\F$ by setting $\chi(0)=0$. Following Greene~\cite{Greene}, we define the {hypergeometric function} over the finite field~$\F$ by
\begin{gather*}
 {}_{n+1}F_n\(\begin{matrix}
 A_0& A_1&\cdots&A_n\\
 & B_1&\ldots& B_n
 \end{matrix};x\) :=\frac{q}{q-1}\sum_{\chi \in \widehat{\Fx}}\CC{A_0\chi}{\chi}\CC{A_1\chi}{B_1\chi}\cdots\CC{A_n\chi}{B_n\chi}\chi(x),
\end{gather*}
where $A_0, A_1,\ldots,A_n$ and $B_1, \ldots, B_n$ are multiplicative characters on~$\F$, and
\begin{gather*}
 \CC AB:=\frac{B(-1)}q J(A,\overline B).
\end{gather*}
Here $J(A,B)$ is the Jacobi sum defined by
\begin{gather*}
 J(A,B)=\sum_{x\in \F}A(x)B(1-x).
\end{gather*}
Equivalently, the hypergeometric functions can be defined inductively by
\begin{gather*}
 {}_{n+1}F_n \(\begin{matrix}
 A_0& A_1&\cdots&A_n\\
 & B_1&\ldots& B_n
 \end{matrix};x\)
 :=\sum_{y\in\F}A_n(y)B_n\overline A_n(1-y) {}_{n}F_{n-1}\(\begin{matrix}
 A_0& A_1&\cdots&A_{n-1}\\
 & B_1&\ldots& B_{n-1}
 \end{matrix};xy\)
\end{gather*}
with
\begin{gather*}
 {}_2F_1\(\begin{matrix}
 A_0& A_1\\
 & B_1
 \end{matrix};x\):= \eps(x)\frac{A_1B_1(-1)}{q}\sum_{y\in \F}A_1(y)B_1\overline A_1(1-y)\overline{A_0}(1-xy),
\end{gather*}
where $\eps$ is the trivial character on $\F^\times$. Being finite fields analogue of the classical hypergeometric functions, the hypergeometric functions over finite fields have close connection to arithmetic geometry and can be used to compute zeta functions of algebraic varieties of hypergeometric type. The cases where the parameters $A_i$ are all the quadratic character $\phi$ and $B_j$ are all the trivial character $\eps$ have particularly drawn a great deal of attention.

Assume that $q=p$ is an odd prime. For convenience, we denote by ${}_{n+1}F_n(x)$ the hypergeometric functions with $A_i=\phi$ and $B_j=\eps$ for all $i$ and $j$. The evaluations of ${}_2F_1(x)$, ${}_3F_2(x)$ at certain rational numbers have been studied by numerous mathematicians such as Barman--Kalita~\cite{BK-certain, BK-elliptic}, Evans--Lam~\cite{Evans-Lam}, Greene--Stanton~\cite{Greene-Stanton}, Koike~\cite{Koike-Aperynumbers, Koike-ellipticcurves}, and Ono~\cite{Ono}. They gave explicit relationship between values of these hypergeometric functions and arithmetic of
elliptic curves. In some cases, the evaluations also have connection with congruence properties of the Ap\'{e}ry numbers. In~\cite{Ahlgren-Ono-hypergeometric, Ahlgren-Ono-CalabiYau}, Ahlgren--Ono expressed the special value ${}_4F_3(1)$ in terms of Fourier coefficients of the Hecke eigenform associated to a certain modular Calabi--Yau manifold. In~\cite{AOP}, Ahlgren, Ono, and Penniston studied the zeta functions of a certain family of $K3$-surfaces, which lead to information about the special values of $_3F_2(x)$ at $x=1, 8, 1/8, -4, -1/4, 64$, and $-1/64$ in terms of the trace of Frobenius on suitable elliptic curves with complex multiplication over~$\Q$~\cite{win3c}. In other words, these values of the hypergeometric functions can also be expressed in terms of Hecke characters on imaginary quadratic number fields.
In \cite{Evans-Greene-Clausentheorem,Evans-Greene-Evaluations}, Evans and Greene obtained evaluations of ${}_3F_2$-hypergeometric functions with more general characters at $x=1/4, -1/8, -1$, etc. Key ingredients in their approach are some transformation formulas of hypergeometric functions over finite fields, such as a finite field analogue
of Clausen's theorem. In recent years, Frechette, Fuselier, Goodson, Lennon, Ono, Papanikolas, and Salerno \cite{Frechette-Ono-Papanikolas, Fuselier-PhD, Fuselier,
 Fuselier-level1, Goodson, Lennon-elliptic, Lennon-modularity, Salerno} investigated
connections between hypergeometric functions over finite fields and elliptic curves, algebraic varieties, and Hecke eigenforms. Very recently, McCarthy--Papanikolas~\cite{McCarthy-Papanikolas} linked the hypergeometric functions to Siegel modular forms.

The focus of this paper will be the evaluations of the hypergeometric function
\begin{gather} \label{equation: intro 0}
 {}_2F_1\(\begin{matrix}
 \phi\psi& \psi\\
 & \phi
 \end{matrix};x\), \qquad x\in \F_p, \qquad x\neq 0, 1,
\end{gather}
for a character $\psi$ of $\widehat{\F^\times_p}$. We will see that if $\psi$ has order~$N$, then the values of this hypergeometric function are related to the number of rational points on a certain superelliptic curve of degree~$N$ or~$2N$, depending on whether $N$ is even or odd, over~$\F_p$. In the case of $N=3,4,6,8,12$, we can explicitly construct morphisms from the superelliptic curve to several elliptic curves over~$\Q$ with complex multiplication. Since the $L$-functions of elliptic curves over~$\Q$ with complex multiplication are known to be equal to the $L$-functions of Hecke characters of the CM field, the values of the hypergeometric function can be expressed in terms of the Hecke characters.

\section{Main statements} \label{sec: Theorem}

To state our main results, we first recall that if $N$ is a positive integer and $p$ is a prime such that $p\equiv 1\mod N$, then $p$ splits completely in $\Q\big(e^{2\pi i/N}\big)$. In other words, if $\fp$ is a~prime of $\Q\big(e^{2\pi i/N}\big)$ lying above~$p$, then $\Z\big[e^{2\pi i/N}\big]/\fp\simeq \F_p$. The integers $0,1,\ldots,p-1$ form a complete set of coset representatives of~$\fp$ in~$\Z\big[e^{2\pi i/N}\big]$. Thus, for a given element~$a$ in~$\F_p$, we may canonically regard it as an element of $\Z\big[e^{2\pi i/N}\big]/\fp$. We now describe our main results.

\begin{Theorem} \label{theorem: N=4} Let $p$ be a prime congruent to~$1$ modulo~$4$, and~$\mathfrak{p}$ be a prime of~$\Z[i]$ lying above~$p$. Let $\psi_{\mathfrak p}$ be the quartic multiplicative character of $(\Z[i]/\fp)^\times\simeq\F_p^\times$ satisfying
 \smash{$\psi_{\mathfrak p}(x)\equiv x^{(p-1)/4}$} $\mod \mathfrak{p}$, for every $x\in \Z[i]$. Then, for $a\in\F_p$ with $a\neq 0, 1$, if one of~$a$ and $1-a$ is not a~square in~$\F_p^\times$, we have
 \begin{gather*}
 {}_2F_1\(\begin{matrix}
 \ol{\psi_\mathfrak p}& {\psi_\mathfrak p}\\
 & \phi
 \end{matrix};a\)=0,
 \end{gather*}
 and if $a=b^2$ for some $b\in\F_p^\times$, then
 \begin{gather*}{}_2F_1\(\begin{matrix}
 \ol{\psi_\mathfrak p}& {\psi_\mathfrak p}\\
 & \phi
 \end{matrix};a\)= -\frac{2}p\phi\(1+b\)\chi({\mathfrak{p}}),
 \end{gather*}
where $\chi$ is the Hecke character associated to the elliptic curve $E\colon y^2=x^3-x$ satisfying $\chi(\mathfrak{p})\in\mathfrak{p}$ for all primes $\mathfrak{p}$ of~$\Z[i]$.
\end{Theorem}

\begin{Remark} The Hecke character $\chi$ in the theorem has the following description. Suppose that the prime ideal $\fp$ is generated by $a+bi$, for some integers~$a$ and~$b$. There exists a unique fourth root $i^k$ of unity such that $a+bi\equiv i^k\mod (1+i)^3$. Then
 \begin{gather*}
 \chi(\fp)=i^k(a+bi).
 \end{gather*}
 Alternatively, we can write $\chi$ as
 \begin{gather*}
 \chi(\mathfrak p)=\begin{cases}
\displaystyle (-1)^{b/2}\JS{-1}a(a+bi),
 &\text{if} \ a \ \text{is odd and} \ b \ \text{is even}, \\
\displaystyle i(-1)^{a/2}\JS{-1}b(a+bi),
 &\text{if}\ a \ \text{is even and}\ b \ \text{is odd}. \end{cases}
 \end{gather*}
\end{Remark}

\begin{Theorem} \label{theorem: N=8}
Let $p$ be a prime congruent to $1$ modulo $8$, $\fP$ a prime of $\Z\big[\sqrt{-2}\big]$ lying above~$p$, and $\fp$ a prime of $\Z\big[e^{2\pi i/8}\big]$ lying above $\fP$. Let $\psi_\fp$ be the character of $\big(\Z\big[e^{2\pi i/8}\big]/\fp\big)^\times\simeq\F_p^\times$ of order~$8$ such that $x^{(p-1)/8}\equiv\psi_\fp(x)\mod\fp$ for all $x\in\Z[\zeta_8]$. Let $\chi$ be the Hecke character associated to the elliptic curve $y^2=x^3+4x^2+2x$ with CM by $\Z\big[\sqrt{-2}\big]$ satisfying $\chi(\fP')\in\fP'$ for all primes~$\fP'$ of~$\Z\big[\sqrt{-2}\big]$. If $a\in\F_p^\times$ is not a square in~$\F_p^\times$, then
 \begin{gather*}
 {}_2F_1\(\begin{matrix}
 \phi\psi_\fp& \psi_\fp\\
 & \phi
 \end{matrix};a\)=0,
 \end{gather*}
 and if $a=b^2$ for some $b\in\F_p^\times$, then
 \begin{gather*}
 {}_2F_1\(\begin{matrix}
 \phi\psi_\fp& \psi_\fp\\
 & \phi
 \end{matrix};a\)
 =-\frac1p(-1)^{(p-1)/8}\big(\ol\psi_\fp^2(1+b) +\ol\psi_\fp^2(1-b)\big)\chi(\fP).
 \end{gather*}
\end{Theorem}

\begin{Remark} The character $\chi$ in the above theorem can be explicitly described as follows. If a~prime~$\fP'$ of~$\Z\big[\sqrt{-2}\big]$ lying above an odd prime is generated by
 $c+d\sqrt{-2}$, then
 \begin{gather*}
 \chi(\fP')=\left(\frac{-2}c\right)\big(c+d\sqrt{-2}\big)\cdot
 \begin{cases}
 (-1)^{d/2}, &\text{if} \ d \ \text{is even}, \\
 -1, &\text{if} \ d \ \text{is odd}. \end{cases}
 \end{gather*}
\end{Remark}

\begin{Theorem} \label{theorem: N=3,6} Let $p$ be a prime congruent to $1$ modulo $6$, and~$\mathfrak{p}$ a~prime ideal of $\Z[\zeta_6]$ \smash{lying} above~$p$. Let $\psi_{\mathfrak p}$ be the multiplicative character of order~$6$ on $(\Z[\zeta_6]/\fp)^\times\simeq\F_p^\times $ satisfying \smash{$ \psi_{\mathfrak p}(x)\equiv x^{(p-1)/6}$} $\mod \mathfrak{p}$, for every $x\in \Z[\zeta_6]$. Let $\chi$ be the Hecke character associated to the elliptic curve $E\colon y^2=x^3+1$ satisfying $\chi(\mathfrak{p})\in\mathfrak{p}$ for all primes $\mathfrak{p}$ of~$Z[\zeta_6]$. If $a\neq 1$ is not a~square in~$\F_p^\times$, we have
 \begin{gather*}
 {}_2F_1\(\begin{matrix}
 {\phi\psi_\mathfrak p^2}& \psi_\mathfrak p^2\\
 & \phi
 \end{matrix};a\)= {}_2F_1\(\begin{matrix}
 {\phi\psi_\mathfrak p}& \psi_\mathfrak p\\
 & \phi
 \end{matrix};a\)=0
 \end{gather*}
 and if $a=b^2$ is a square in $\F_p^\times$, then
 \begin{gather*}
 {}_2F_1\(\begin{matrix}
 {\phi\psi_\mathfrak p^2}& \psi_\mathfrak p^2\\
 & \phi
 \end{matrix};a\)= -\frac{\phi(-1)}p
 \big(\psi_\fp^2(1+b)+\psi_\fp^2(1-b)\big)\chi(\fp),\\
 {}_2F_1\(\begin{matrix}
 {\phi\psi_\mathfrak p}& \psi_\mathfrak p\\
 & \phi
 \end{matrix};a\)= -\frac{1}{p}
 \big(\ol\psi_\fp^2(1+b)+\ol\psi_\fp^2(1-b)\big)\chi(\fp).
 \end{gather*}
\end{Theorem}

\begin{Remark} The Hecke character $\chi$ can be explicitly given as follows. Suppose that $\fp$ is generated by $\alpha\in\Z[\zeta_6]$. There exists a unique sixth root $\zeta_6^k$ of unity such that $\alpha\equiv\zeta_6^k\mod 2\sqrt{-3}$. Then we have
 \begin{gather*}
 \chi(\fp)=\zeta_6^k\alpha.
 \end{gather*}
\end{Remark}

\begin{Theorem} \label{theorem: N=12} Let $p$ be a prime congruent to $1$ modulo $12$, $\fP$ a prime of $\Z[i]$ lying above $p$, and $\fp$ a prime of $\Z\big[e^{2\pi i/12}\big]$ lying above $\fP$. Let $\psi_\fp$ be the character of order $12$ in $\widehat{\F_p^\times}$ such that $x^{(p-1)/12}\equiv\psi_\fp(x)\mod\fp$ for all $x\in\Z[\zeta_{12}]$. Let~$\chi$ be the Hecke character associated to the elliptic curve $y^2=x^3-3x$ satisfying $\chi(\fP')\in\fP'$ for all primes~$\fP'$ of~$\Z[i]$. If $a\in\F_p^\times$ is not a square in~$\F_p^\times$, then
 \begin{gather*}
 {}_2F_1\(\begin{matrix}
 \phi\psi_\fp& \psi_\fp\\
 & \phi
 \end{matrix};a\)=0,
 \end{gather*}
 and if $a=b^2$ for some $b\in\F_p^\times$, then
 \begin{gather*}
 {}_2F_1\(\begin{matrix}
 \phi\psi_\fp& \psi_\fp\\
 & \phi
 \end{matrix};a\)
 =-\frac1p(-1)^{(p-1)/12}\big(\ol\psi_\fp^2(1+b)+\ol\psi_\fp^2(1-b)\big)\chi(\fp).
 \end{gather*}
\end{Theorem}

\begin{Remark} The elliptic curve $y^2=x^3-3x$ is isomorphic to the elliptic curve $y^2=x^3-x$ in Theorem~\ref{theorem: N=4} over $\Q(\sqrt[4]3)$. From this fact, we may deduce an expression for~$\chi$ as follows. Assume that $p$ is a prime congruent to~$1$ modulo $4$ and $\fp$ is a prime of $\Z[i]$ lying above $p$ generated by~$a+bi$. There exists a unique fourth root of unity $i^k$ such that
 \begin{gather*}
 3^{(p-1)/4}\equiv i^k\mod \fp
 \end{gather*}
 Then
 \begin{gather*}
 \chi(\fp)=\begin{cases}
 \displaystyle i^{-k}(-1)^{b/2}\JS{-1}a(a+bi),
 &\text{if} \ a \ \text{is odd and} \ b \ \text{is even}, \\
 \displaystyle i^{1-k}(-1)^{a/2}\JS{-1}b(a+bi),
 &\text{if} \ a \ \text{is even and} \ b \ \text{is odd}. \end{cases}
 \end{gather*}
\end{Remark}

\section{Preliminary lemmas}

For any characters $A, B, C \in \widehat{\F_q^\times}$ and any element $z\in \F_q$, the ${}_2F_1$-hypergeometric function over~$\F_q$ can be written as
\begin{gather*}\begin{split}&
 \hgq ABCz= \eps(z)\frac{BC(-1)}{q}\sum_{x\in\F_q}B(x)C\ol B(1-x)\ol A(1-zx)\\
& \hphantom{\hgq ABCz}{} = \eps(z)\frac{BC(-1)}{q}\sum_{x\in\F_q}A\ol C(x)C\ol B(x-1)\ol A(x-z).\end{split}
\end{gather*}
Thus, if $\eta \in \widehat{\F_q^\times}$ and $z\neq 0$, we have
\begin{gather} \label{equation: first step}
 {}_2F_1\(\begin{matrix}
 \phi\eta& \eta\\
 & \phi
\end{matrix};z\)=\frac{\phi\eta(-1)}q\sum_{x\in\F_q} \eta(x)\phi\ol \eta(x-1)\phi\ol \eta(x-z),
\end{gather}
 and
\begin{gather}
 {}_2F_1\(\begin{matrix}
 \phi\eta& \eta\\
 & \phi
 \end{matrix};z\)
 =\frac{q}{q-1}\sum_{\chi\in\widehat{\F_q^\times}}\CC{\phi\eta\chi}
 {\chi}\CC{\eta\chi}{\phi\chi}\chi(z)
\overset{\chi\mapsto \phi\chi}{=} \frac{q}{q-1}\sum_{\chi\in\widehat{\F_q^\times}}\CC{\eta\chi}{\phi\chi} \CC{\phi\eta\chi}{\chi}\phi\chi(z)\nonumber\\
 \hphantom{{}_2F_1\(\begin{matrix}
 \phi\eta& \eta\\
 & \phi
 \end{matrix};z\)}{} =
 \phi(z) {}_2F_1\(\begin{matrix}
 \phi\eta& \eta\\
 & \phi
 \end{matrix};z\).\label{eqn: trace}
\end{gather}
The value $\hgq{\phi\eta}{\eta}{\phi}z$ hence equals to zero if $\phi(z)=-1$.

In order to evaluate $\hgq{\phi\eta}{\eta}{\phi}z$ when $\phi(z)=1$, we need to consider the character sum
 \begin{gather*}
 F_\eta(z)=\sum_{x\in\F_q} \eta(x)\phi\ol \eta(x-1)\phi\ol \eta(x-z),
\end{gather*}
on the right-hand side of \eqref{equation: first step} so that
\begin{gather} \label{equation: prelim 1}
 \hgq{\phi\eta}\eta\phi{z}=\frac{\phi\eta(-1)}qF_\eta(z).
\end{gather}
It is easy to see that
\begin{gather*}
 F_\eta(1)=\sum_{x\in\F_q} \eta(x)\phi\ol \eta(x-1)\phi\ol \eta(x-1)=J\big(\eta,\ol\eta^2\big),
\end{gather*}
and if $\eta^2=\eps$ then
\begin{gather*}
 F_\eta(z)=-1-\phi(z), \qquad z\neq 0,1.
\end{gather*}
By a simple change of variables, one can derive a relation between the character sums~$F_{\eta}(z)$ and~$F_{\phi\ol\eta}(z)$ as follows.

\begin{Proposition}\label{prop: eta-phieta} Let $\F_q$ be the finite field of $q$ elements and $\eta$ a~character of order~$N$ with $N \,|\, (q-1)$. For any $z\in\F_q$ with $z\neq 1$, we have
 \begin{gather*}
 F_{\eta}(z)=\ol\eta(1-z)^2F_{\phi\ol\eta}(z).
 \end{gather*}
\end{Proposition}
\begin{proof} For $z\neq 1$, replacing $x$ by $(x-1)/(x-z)$, we obtain
 \begin{align*}
 F_\eta(z)& =\sum_{x\in \F} \eta(x)\phi\ol \eta\((x-1)(x-z)\)\\
& =\sum_{x\in \F} \eta\(\frac{x-1}{x-z}\)\phi\ol \eta\(\(\frac{x-1}{x-z}-1\)\(\frac{x-1}{x-z}-z\)\)\\
& =\phi\eta(-1)\ol\eta(1-z)^2\sum_{x\in \F} \eta\((x-z)(x-1)\)\phi\ol\eta(x-(1+z))\\
& =\phi\eta(-1)\ol\eta(1-z)^2\sum_{x\in \F}
 \eta\((x+1)(x+z)\)\phi\ol\eta(x)\\
& = \ol\eta(1-z)^2\sum_{x\in \F} \eta\((1-x)(z-x)\)\phi\ol\eta(x).
 \end{align*}
This gives the identity $ F_\eta(z)=\ol\eta^2(1-z)F_{\phi\ol\eta}(z)$.
\end{proof}

Assume that $\eta$ is a character of order $N>2$ and $z\neq 1$. When $N$ is even,
we have $\phi=\eta^{N/2}$ and hence
\begin{align*}
 F_\eta(z)&=\ol\eta(1-z)^2F_{\phi\overline\eta}(z)
 =\ol\eta(1-z)^2\sum_{x\in\F_q}\phi\ol\eta(x)\eta(x-1)\eta(x-z) \\
 &=\ol\eta(1-z)^2\sum_{x\in\F_q}\eta\big(x^{N/2-1}(x-1)(x-z)\big).
\end{align*}
When $N$ is odd, $\phi\ol\eta$ is a character of order $2N$ and $\eta=(\phi\ol\eta)^{N-1}$. Consequently, we have
\begin{gather*}
 F_\eta(z)=\sum_{x\in\F_q}\phi\ol\eta\big(x^{N-1}(x-1)(x-z)\big).
\end{gather*}

Let $\eta$ be a character of $\F_q^\times$ of order $N$. Set
 \begin{gather*}
 N'=\begin{cases}
 N, &\text{if} \ N \ \text{is even}, \\
 2N, &\text{if} \ N \ \text{is odd}, \end{cases} \qquad
 \eta'=\begin{cases}
 \eta, &\text{if} \ N \ \text{is even}, \\
 \phi\ol\eta, &\text{if} \ N \ \text{is odd}, \end{cases}
 \end{gather*}
 and
 \begin{gather} \label{equation: Seta}
 S_{\eta'}(z)=\sum_{x\in\F_q}\eta'(x^{N'/2-1}(x-1)(x-z)).
 \end{gather}

We summarize the short discussion above in the following lemma.

\begin{Lemma} \label{lemma: reduce} Let $\eta$ be a character of $\F_q^\times$ of order $N>2$, and $z\in \F_q$ with $z\neq 0, 1$. Then we have
 \begin{gather*}
 \hgq{\phi\eta}\eta\phi z=\frac{\phi\eta(-1)}q\begin{cases}
 \ol\eta(1-z)^2S_{\eta'}(z),
 &\text{if} \ N \ \text{is even}, \\
 S_{\eta'}(z), &\text{if} \ N \ \text{is odd}. \end{cases}
 \end{gather*}
\end{Lemma}

From the lemma, we see that the evaluation of our hypergeometric functions reduces to that of $S_\eta(z)$ for a character of even order~$N$. It is natural to consider the following families of curves
\begin{gather} \label{equation: curve C}
 C_{N,a,c}\colon \
 \begin{cases}
 y^N=cx^{N/2-1}(x-1)(x-a),& \mbox{if $N$ is even}, \\
 y^{2N}=cx(x-1)^{N-1}(x-a)^{N-1},& \mbox{if $N$ is odd}
 \end{cases}
\end{gather}
over $\Q$. We will see that our sums $S_\eta(a)$ appear in the $L$-functions of $C_{N,a,c}$. By varying $c$, we are able to determine~$S_\eta(a)$.

The decomposition of the Jacobian variety of $C_{N,a,c}$ will give us information about $S_\eta(a)$ with~$\eta$ of order~$N$. For this purpose, we first consider the quotient curves of $C_{N,a,c}$ arising from the automorphisms of $C_{N,a,c}$.

\begin{Lemma} Let $a$ and $c$ be nonzero rational numbers and $b$ be a square root of $a$ in $\overline\Q$. Then the curve $C_{N,a,c}$ admits the automorphisms
\begin{gather*}
 \sigma\colon \ (x,y)\longmapsto(x,\zeta_{N'}y), \qquad \tau\colon \ (x,y)\longmapsto (a/x,by/x )
\end{gather*}
defined over $\Q(\zeta_{N'},b)$. Here for a positive integer $M$, we let $\zeta_M=e^{2\pi i/M}$.

Moreover, $\sigma$ and $\tau$ commute with each other.
\end{Lemma}

\begin{Lemma} \label{lemma: morphism} Let $a$ and $c$ be nonzero rational numbers. Let~$b$ denote a square root of $a$ in~$\overline\Q$. For even $N$, the quotient of~$C_{N,a,c}$ by~$\gen\tau$ gives rise to the morphism $C_{N,a,c}\to cY^2=X^{N/2+1}+c(1+b)^2X$ with
 \begin{gather*}
 (X,Y)=\left(\frac{y^2}x,\frac{y(x+b)}x\right).
 \end{gather*}
 For odd $N$, the corresponding morphism is $C_{N,a,c}\to cY^2=cX^N+(1+b)^2$ with
 \begin{gather*}
 (X,Y)=\left(\frac{(x-1)(x-a)}{y^2},\frac{(x+b)(x-1)^{(N-1)/2}(x-a)^{(N-1)/2}}
 {y^N}\right).
 \end{gather*}
\end{Lemma}

These two lemmas can be verified by direct computation.

For a positive even integer $N$ and nonzero integers $c$ and $d$, let $D_{N,c,d}$ denote the curve
\begin{gather*}
 D_{N,c,d}\colon \ cy^2=x^{N/2+1}+cdx.
\end{gather*}
The following lemma expresses the number of rational points of
$D_{N,c,d}$ over a given finite field in terms of Jacobi sums.

\begin{Lemma}[{\cite[Chapter 6]{Berndt-Evans-Williams}}]\label{lemma: DNc} Let $N$ be a~positive even integer and $q=p^r$ be an odd prime power such that $N\,|\, (q-1)$. Assume that $c$ is an integer with $p\nmid c$. When $N$ is of the form $4n$, we have
 \begin{gather*}
 \# D_{N,c,d}(\F_q)=q+1+\phi(d)
 \sum_{\psi^N=\varepsilon,\, \psi^{N/2}\neq\varepsilon}\psi(-cd)J(\phi,\psi).
 \end{gather*}
 When $N$ is of the form $4n+2$, we have
 \begin{gather*}
 \# D_{N,c,d}(\F_q)=q+1+\phi(c)+\phi(d)
 \sum_{\psi^N=\varepsilon,\, \psi^{N/2}\neq\varepsilon}\psi(-cd)J(\phi,\psi).
 \end{gather*}
\end{Lemma}

\begin{proof} When $N=4n$, the curve $D_{N,c,d}$ has one point defined over~$\F_q$ at infinity. Thus,
 \begin{gather*}
 \# D_{N,c,d}(\F_q)=1+\sum_{x\in\F_q}\big(1+\phi\big(c\big(x^{2n+1}+cdx\big)\big)\big)
=q+1+\phi(c)\sum_{x\in\F_q}\phi(x)\phi\big(x^{2n}+cd\big).
 \end{gather*}
 According to Theorem~6.1.14 of~\cite{Berndt-Evans-Williams} (to be
 more precisely, the second to the last line of the proof), we have
 \begin{gather*}
 \sum_{x\in\F_q}\phi(x)\phi(x^{2n}+cd)
 =\phi(cd)\sum_{\psi^N=\varepsilon, \, \psi^{N/2}\neq\varepsilon}\psi(-cd)J(\phi,\psi),
 \end{gather*}
 which proves the case $N=4n$. The proof of the other case is almost the same, the only differences being that when $N=4n+2$, there are $1+\phi(c)$ points over $\F_q$ at infinity.
\end{proof}

We are now ready to evaluate our hypergeometric functions.

\section[Case $N=4$]{Case $\boldsymbol{N=4}$}\label{sec: N=4}

In this section, we consider the case where the character $\psi$ in~\eqref{equation: intro 0} has order $4$. Let $\sigma$ be the non-trivial Galois element of~$\operatorname{Gal}(\Q(i)/\Q)$. For a prime~$p$ congruent to $1$ modulo $4$ and a~prime~$\fp$ of~$\Z[i]$ lying over $p$, we let $\psi_\fp$ be the character of order $4$ of~$\Z[i]/\fp$ satisfying \smash{$x^{(p-1)/4}\equiv\psi_\fp(x)$} $\mod\fp$ for all~$x\in\Z[i]$. Also, let~$E$ be the elliptic curves $y^2=x^3-x$ over~$\Q$. These are elliptic curves with CM by~$\Z[i]$. We let $\chi$ be the Hecke character of~$\Z[i]$ associated to~$E$ (that is, $L(E,s)=L(\chi,s)$) satisfying $\chi(\fP)\in \fP$ for all prime ideals~$\fP$ of~$\Z[i]$.

\begin{Lemma} \label{lemma: same L 4} Let $c$, $d_1$, and $d_2$ be nonzero integers. The elliptic curves $cy^2=x^3+cd_1x$ and $cy^2=x^3+cd_2x$ defined over $\Q$ have the same $L$-function if and only if $d_1/d_2$ is a rational number of the form $e^4$ or $-4e^4$ for some rational number $e$.
\end{Lemma}

This Lemma follows from the fact that the isogeny class defined over $\Q$ of the elliptic curve $y^2=x^3+dx$ with some $d\in\Q$ contains the elliptic curves given by the equations $y^2=x^3+Dx$ and $y^2=x^3-4Dx$, where with $D/d$ is a fourth power of a rational number. For more detail, please see {\cite[Chapters~IX.7 and~X.6]{Silverman}} for example.

\begin{proof}[Proof of Theorem \ref{theorem: N=4}]
 In equation (\ref{eqn: trace}), we have seen that if $a$ is not a
 square in $\F_p^\times$, then the value of the $_2F_1$-hypergeometric
 function is zero. Furthermore, the value also vanishes when $1-a$ is
 not a square in $\F_p^\times$ due to Proposition~\ref{prop:
 eta-phieta}.
 Therefore, in the rest of proof, we will only consider the case where
 $a$ and $1-a$ are both squares in $\F_p^\times$.

 Let $c$ be a nonzero squarefree integer and $b$ be a rational number
 such that $b^2=a$ in $\F_p$. The curve $C_{4,a,c}$ has genus $3$ and it is a $2$-fold cover of the following $3$ elliptic curves
 \begin{alignat*}{3}
 & C_{2,a,c}\colon\ && y^2=cx(x-1)(x-a),&\\
 & E_{+,b,c}\colon \ && cy^2=x^3+c(1+b)^2x,&\\
 & E_{-,b,c}\colon \ && cy^2=x^3+c(1-b)^2x.&
 \end{alignat*}
 The morphisms are given by
 \begin{alignat*}{3}
 & C_{4,a,c}\longrightarrow C_{2,a,c}\colon \ && (x,y)\mapsto\big(x,y^2\big),&\\
& C_{4,a,c}\longrightarrow E_{\pm,b,c}\colon \ &&
 (x,y)\mapsto\left(\frac{y^2}x,\frac{y(x\pm b)}x\right)&
 \end{alignat*}
 (see Lemma \ref{lemma: morphism}). We first claim that $C_{2,a,c}$ is not isogenous to any of~$E_{\pm,b,c}$. Indeed, the elliptic curves $E_{\pm,b,c}$ have~CM by~$\Z[i]$. If $C_{2,a,c}$ is isogenous to $E_{\pm,b,c}$, then $C_{2,a,c}$ has CM by either $\Z[i]$ or $\Z[2i]$ and hence has a~$j$-invariant $12^3$ or $66^j$. (Elliptic curves with CM by $\Z[ri]$ are not defined over $\Q$ for $r>2$.) The only rational numbers $a$ such that $C_{2,a,c}$ has one of these $j$-invariants are~$-1$,~$2$, and~$1/2$, but neither of them is a~square in~$\Q$. Thus, $C_{2,a,c}$ cannot be isogenous to~$E_{\pm,b,c}$.

Moreover, we may assume that $b$ is a number such that $(1+b)^2/(1-b)^2$ is not of the form~$e^4$ or~$-4e^4$ for some $e\in\Q$. Thus, these three elliptic curves are not isogenous to each other over~$\Q$. The $L$-function of $C_{4,a,c}$ over~$\Q$ therefore is equal to
 \begin{gather*}
 L( C_{4,a,c}/\Q,s)=L( C_{2,a,c}/\Q,s)L( E_{+,b,c}/\Q,s)L( E_{-,b,c}/\Q,s).
 \end{gather*}
 In other words, for almost all primes $p$, we have
 \begin{gather*}
 a_{4,p}=a_{2,p}+b_{+,p}+b_{-,p},
 \end{gather*}
 where $b_{\pm,p}$, $a_{2,p}$, and $a_{4,p}$ are the coefficients in the local $L_p$-functions, the reciprocal of the Euler $p$-factor of the $L$-functions,
 \begin{gather*}
 L_p(E_{\pm,b,c},s)=1-b_{\pm, p}p^{-s}+p^{1-2s},\\
 L_p(C_{2,a,c},s)=1-a_{2, p}p^{-s}+p^{1-2s},
 \end{gather*}
 and
 \begin{gather*}
 L_p(C_{4,a,c},s)=1-a_{4, p}p^{-s}+\cdots
 \end{gather*}
of these four curves at $p$, respectively. On the other hand, if the curve $C_{4,a,c}$ has good reduction at a~prime~$p$ with $p\equiv 1 \mod 4$, then
\begin{gather*}
 \# C_{4,a,c}(\F_p)=
 p+1+\sum_{\substack{\chi\in\widehat{\F_p^\times},\, \chi^4=\eps,\,
 \chi\neq\varepsilon}} \sum_{x\in\F_p} \chi\(cx(x-1)(x-a)\) = p+1-a_{4,p},\\
 \# C_{2,a,c}(\F_p)=p+1+\sum_{x\in\F_p} \phi\(cx(x-1)(x-a)\)=p+1-a_{2,p},
\end{gather*}
and
\begin{gather*}
 \# E_{\pm,b,c}(\F_p) =p+1-b_{\pm,p},
\end{gather*}
where $\varepsilon$ denotes the trivial character of~$\widehat{\F_p^\times}$. From these identities, we deduce that
\begin{gather*}
 p+1-a_{4, p}=p+1-a_{2,p}+\sum_{\substack{\chi^4=\eps,\, \chi^2\neq\eps}}\sum_x \chi\(cx(x-1)(x-a)\),
\end{gather*}
and thus
\begin{align*}
 -b_{+,p}-b_{-,p}
 &=\sum_{\substack{\chi^4=\eps,\,\chi^2\neq\eps}}\sum_x
 \chi\(cx(x-1)(x-a)\)\\
 &=\eta(c)S_\eta(a)+\ol\eta(c)S_{\ol\eta}(a)
 =\eta(c)S_\eta\big(b^2\big)+\ol\eta(c)S_{\ol\eta}\big(b^2\big),
\end{align*}
where $\eta$ is a multiplicative character of order $4$ of~$\F_p^\times$ and $S_\eta$ is the character sum defined in~\eqref{equation: Seta}. On the other hand, by Lemma~\ref{lemma: DNc}
\begin{gather*}
 -b_{+,p}-b_{-,p} =\eta(-1)\(\phi(1+b)+\phi(1-b)\)
 \(\eta(c)J(\phi,\eta)+\ol\eta(c)J(\phi,\ol\eta)\).
\end{gather*}
Since $1-a=1-b^2$ is assumed to be a square in $\F_p^\times$, one has
\begin{gather*}
 \phi(1-a)=\phi(1-b)\phi(1+b)=1
\end{gather*}
and one can simplify the identity above as
\begin{gather*}
 \eta(c)S_{\eta}\big(b^2\big)+\ol\eta(c)S_{\ol\eta}\big(b^2\big)
 =2\eta(-1)\phi(1+b)\(\eta(c)J(\phi,\eta)+\ol\eta(c)J(\phi,\ol\eta)\).
\end{gather*}
Now we apply this identity with $c=1$ and a nonzero integer $c$ such that $\eta(c)=i$, respectively, and obtain
\begin{gather*}
 iS_{\eta}\big(b^2\big)-iS_{\ol\eta}\big(b^2\big)
 =2\eta(-1)\phi(1+b)\(iJ(\phi,\eta)-iJ(\phi,\ol\eta)\), \\
 S_{\eta}\big(b^2\big)+S_{\ol\eta}\big(b^2\big)
 =2\eta(-1)\phi(1+b)\(J(\phi,\eta)+J(\phi,\ol\eta)\),
\end{gather*}
from which we conclude that
\begin{gather} \label{equation: hello}
 S_{\eta}\big(b^2\big)=2\phi(1+b)\eta(-1)J(\phi,\eta).
\end{gather}

As the last piece of information needed for a proof of the theorem, we recall that when $\eta=\psi_\fp$, one has
\begin{gather} \label{equation: hello again}
 J(\phi,\psi_\fp)=-\chi(\fp),
\end{gather}
where $\chi$ is the Hecke character on $\Z[i]$ specified in the theorem (see \cite{Berndt-Evans-Williams, Ireland-Rosen, Silverman-advence, Weil}). Com\-bining~\eqref{equation: prelim 1}, \eqref{equation: hello}, \eqref{equation: hello again}, and Lemma~\ref{lemma: reduce}, we obtain
\begin{align*}
 \hgq{\overline\psi_\fp}{\psi_\fp}{\phi}a
 & =\hgq{\phi\psi_\fp}{\psi_\fp}{\phi}a =\frac{\phi\psi_\fp(-1)}p F_{\psi_\fp}(a)
 =\frac{\psi_\fp(-1)}p\overline\psi_\fp(1-a)^2S_{\psi_\fp}(a) \\
 &=\frac{2}p\phi(1+b)J(\phi,\psi_\fp)=-\frac2p\phi(1+b)\chi(\fp),
\end{align*}
where the deduction from the second line to the third line uses the assumptions that $1-a$ is a~square in $\F_p$ and $\psi_\fp$ is a~character of order~$4$. This completes the proof of the theorem. \end{proof}

\section[Case $N=8$]{Case $\boldsymbol{N=8}$}

In this section, we let $\zeta_8=e^{2\pi i/8}$ and $\sigma_j\in\operatorname{Gal}(\Q(\zeta_8)/\Q)$ be the Galois element satisfying $\sigma_j(\zeta_8)=\zeta_8^j$. For a prime $p$ congruent to $1$ modulo $8$ and a prime~$\fp$ of $\Z[\zeta_8]$ lying over $p$, we let~$\psi_\fp$ be the character of order $8$ in $\wFx$ satisfying $x^{(p-1)/8}\equiv\psi_\fp(x)\mod\fp$ for all $x\in\Z$. Also, let~$E_1$ and $E_2$ be the elliptic curves $y^2=x^3+4x^2+2x$ and $y^2=x^3-4x^2+2x$ over~$\Q$. These are elliptic curves with CM by $\Z\big[\sqrt{-2}\big]$. We let $\chi$ be the Hecke character of $\Z\big[\sqrt{-2}\big]$ associated to~$E_1$ (that is, $L(E_1,s)=L(\chi,s)$) satisfying $\chi(\fP)\in\fP$ for all prime ideals~$\fP$ of~$\Z\big[\sqrt{-2}\big]$.

\begin{Lemma} \label{lemma: L of D81}
 We have the following properties about the $L$-function of $D_{8,1,1}\colon y^2=x^5+x$.
\begin{enumerate}\itemsep=0pt
\item[$1.$] We have $L(D_{8,1,1}/\Q,s)=L(E_1/\Q,s)L(E_2/\Q,s)$.
\item[$2.$] Let $p$ be a prime congruent to $1$ modulo $8$. Let $\mathfrak P$ be a prime of $\Z\big[\sqrt{-2}\big]$ lying over~$p$ and~$\fp$ a prime of $\Z[\zeta_8]$ lying over $\fP$. We have
 \begin{gather*}
 J(\phi,\psi_\fp)=J\big(\phi,\psi_\fp^3\big)=-\psi_\fp(-1)\chi(\mathfrak P), \\
 J\big(\phi,\psi_\fp^5\big)=J\big(\phi,\psi_\fp^7\big)=-\psi_\fp(-1)\chi( \overline{\mathfrak P}).
 \end{gather*}
\end{enumerate}
\end{Lemma}

\begin{proof} Part (1) is proved in~\cite{Hashimoto-Long-Yang}. We now prove part~(2).

To prove $J(\phi,\psi_\fp)=J\big(\phi,\psi_\fp^3\big)$, we make a change of variable $x\mapsto x/(x-1)$ in the definition of $J(\phi,\psi_\fp)$ and obtain
 \begin{align}
 J(\phi,\psi_\fp)&=\sum_{x\neq0,1}\phi(x)\psi_\fp(1-x)
 =\sum_{x\neq0,1}\phi\left(\frac x{x-1}\right)
 \psi_\fp\left(-\frac1{x-1}\right) \nonumber\\
 &=\psi_\fp(-1)\sum_{x\neq0,1}\phi(x)\psi_\fp(x-1)^{4-1}
 =J\big(\phi,\psi_\fp^3\big). \label{equation: N=8 temp}
 \end{align}
 Likewise, we have $J\big(\phi,\psi_\fp^5\big)=J\big(\phi,\psi_\fp^7\big)$. Now from Stickelberger's theorem (see \cite[Chapter~14]{Ireland-Rosen}), we know that
 \begin{gather*}
 J(\phi,\psi_\fp)=J\big(\phi,\psi_\fp^3\big)\in\fp_1\fp_3=\mathfrak P
 \Z[\zeta_8], \qquad J\big(\phi,\psi_\fp^5\big)=J\big(\phi,\psi_\fp^7\big)\in\fp_5\fp_7
 =\overline{\mathfrak P}\Z[\zeta_8].
 \end{gather*}
 Recalling that $\big|J\big(\phi,\psi_\fp^j\big)\big|^2=p$, we see that
 \begin{gather*}
 J(\phi,\psi_\fp)=J\big(\phi,\psi_\fp^3\big)=u\big(a+b\sqrt{-2}\big)
 \end{gather*}
for some root of unity $u$ in $\Z[\zeta_8]$, where $a+b\sqrt{-2}$ is an element in $\Z\big[\sqrt{-2}\big]$ that generates $\fP$. In view of
 \begin{gather} \label{equation: Galois on J}
 \sigma_j(J(\phi,\psi_\fp))=\sum_{x}\phi(x) \psi_\fp(1-x)^j=J\big(\phi,\psi_\fp^j\big),
 \end{gather}
 where $\sigma_j$, $j=1,3,5,7$, is the element in
 $\operatorname{Gal}(\Q(\zeta_8)/\Q)$ with $\sigma_j(\zeta_8)=\zeta_8^j$,
 we find that $u$ must be either $1$ or $-1$.

 Now by Lemma \ref{lemma: DNc},
 \begin{gather*}
 \# D_{8,1,1}(\F_p)=p+1+\sum_{j=1,3,5,7}\psi_\fp(-1)J\big(\phi,\psi_\fp^j\big).
 \end{gather*}
 By Part (1), this is equal to
 \begin{gather*}
 p+1-\chi(\fP)-\chi\big(\ol\fP\big)-\chi'(\fP)-{\chi'}\big(\ol\fP\big),
 \end{gather*}
 where $\chi'$ is the Hecke character attached to $E_2$. Since $E_2$ is the twist of $E_1$ by $-1$ and $p$ is assumed to be congruent to $1$ modulo~$8$, we actually have
 $\chi'(\fP)=\chi(\fP)$, and hence
 \begin{gather*}
 2\psi_\fp(-1)\big(J(\phi,\psi_\fp)+J\big(\phi,\psi_\fp^7\big)\big) =-2\chi(\fP)-2\chi\big(\ol\fP\big).
 \end{gather*}
 From this, we see that
 $\chi(\mathfrak P)+\psi_\fp(-1)J(\phi,\psi_\fp)
 =-\chi\big(\overline{\mathfrak P}\big)-\psi_\fp(-1)J\big(\phi,\ol\psi_\fp\big)
 $
 lies in $\fP\ol\fP=p\Z\big[\sqrt{-2}\big]$, i.e.,
 \begin{gather*}
 \chi(\mathfrak P)+\psi_\fp(-1)J(\phi,\psi_\fp)
 =-\chi\big(\overline{\mathfrak P}\big)-\psi_\fp(-1)J\big(\phi,\ol\psi_\fp\big)
 =p\alpha
 \end{gather*}
 for some $\alpha\in\Z\big[\sqrt{-2}\big]$. As the absolute values of~$J(\phi,\psi_\fp)$ and $\chi(\fP)$ are both $\sqrt p$, the only possibility is that $\alpha=0$ and we have $J(\phi,\psi_\fp)=-\psi_\fp(-1)\chi(\mathfrak P)$ and $J\big(\phi,\ol\psi_\fp\big)=-\psi_\fp(-1)\chi\big(\overline{\mathfrak P}\big)$. This proves the lemma.
\end{proof}

\begin{Lemma} \label{lemma: same L 8} Let $c$, $d$, $d_1$, and $d_2$ be nonzero integers. If $cd$ is not of the form~$\pm e^4$ or~$\pm 4e^4$ for some integer~$e$, then
 for any prime $\ell$ not dividing $2cd$, the Galois representation~$\rho_{\ell}$ of $\operatorname{Gal}\big(\overline\Q/\Q\big)$ associated to the curve $D_{8,c,d}\colon cy^2=x^5+cdx$ is irreducible.

Also, the $L$-functions of the curves $cy^2=x^5+cd_1x$ and $cy^2=x^5+cd_2x$ are equal if and only if $d_1/d_2$ or $d_1/c^2d_2^3$ is a rational number of the form $e^8$ or~$16e^8$ for some nonzero rational number $e$.
\end{Lemma}

\begin{proof} For a prime $p$ congruent to $1$ modulo $8$, we know from Lemmas~\ref{lemma: DNc} and~\ref{lemma: L of D81} that the reciprocal of the $p$-factor of $L(D_{8,c,d}/\Q,s)$ is
 \begin{gather*}
 \big(1-\phi(c)\psi_\fp(cd)\chi(\fP)p^{-s}\big) \big(1-\phi(c)\psi_\fp(cd)^3\chi(\fP)p^{-s}\big) \\
 \qquad{}\times \big(1-\phi(c)\psi_\fp(cd)^5\chi(\overline{\fP})p^{-s}\big)
 \big(1-\phi(c)\psi_\fp(cd)^7\chi(\overline{\fP})p^{-s}\big),
 \end{gather*}
where $\fP$ is a prime of $\Z\big[\sqrt{-2}\big]$ lying above $p$, and~$\fp$ is the prime of~$\Z[\zeta_8]$ lying above~$\fP$. We claim that under the assumption of the lemma, there exists a prime~$p$ congruent to~$1$ modulo~$8$ such that $\psi_\fp(cd)\neq\pm 1$. Then the first part of the lemma follows since $\psi_\fp(cd)\chi(\fP)$ is an algebraic number of degree~$4$.

Assume that $\psi_\fp(cd)=\pm 1$ for all primes $p$ congruent to~$1$ modulo~$8$. Then \smash{$(cd)^{(p-1)/4}\equiv 1$} $\mod p$ for all such primes. It follows that the reduction modulo $p$ of the polynomial $x^4-cd$ splits into a product of~$4$ linear factors over~$\F_p$. Let $E$ be the splitting field of $x^4-cd$ over~$\Q$ and~$G$ be the Galois group of~$E$ over~$\Q$. Then by the Chebotarev density theorem, we have~$|G|\le 4$. Thus, $cd$ must be of the form $\pm e^4$, $\pm 4e^4$, or~$e^2$ for some integer~$e$. However, the last possibility $cd=e^2$ cannot occur because we can always find a prime~$p$ congruent to~$1$ modulo $8$ such that~$e$ is a quadratic nonresidue modulo~$p$. Thus, we conclude that if~$cd$ is not of the form~$\pm e^4$ or~$\pm 4e^4$, then there exists a prime $p$ congruent to~$1$ modulo~$8$ such that $\psi_\fp(cd)\neq\pm 1$. This proves the first part of the lemma.

Assume that the $L$-functions of $D_{8,c,d_1}$ and $D_{8,c,d_2}$ are equal. By Lemmas~\ref{lemma: DNc} and~\ref{lemma: L of D81}, this implies that for all primes~$p$ congruent to $1$ modulo $8$ with $p\nmid cd_1d_2$ and all characters~$\psi$ of order~$8$ in~$\wFx$,
 \begin{gather*}
 \psi(cd_1)+\psi(cd_1)^3=\psi(cd_2)+\psi(cd_2)^3.
 \end{gather*}
As $\psi$ takes values in the set of eighth roots of unity, we must have $\psi(cd_1)=\psi(cd_2)$ or $\psi(cd_1)=\psi(cd_2)^3$, which implies that $d_1x^8-d_2$ or $d_1x^8-c^2d_2^3$ splits completely modulo~$p$. Again, we can use the Chebotarev density theorem to conclude that either $d_1/d_2$ or $d_1/c^2d_2^3$ is of the form $e^8$ or $16e^8$ for some nonzero rational number~$e$.

Conversely, if $d_1/d_2=e^8$ for some nonzero rational number $e$, then the map $(x,y)\mapsto\big(e^2x,e^5y\big)$ defines an isomorphism from $D_{8,c,d_1}$ to $D_{8,c,d_2}$ over $\Q$. If $d_1/d_2=16e^8$, then $(x,y)\to\big(2e^2x,4\sqrt2e^y\big)$ defines an isomorphism over $\Q\big(\sqrt 2\big)$. Since the curves $D_{8,c,d_1}$ and $D_{8,c,d_2}$ have real multiplication by $\Q\big(\sqrt 2\big)$, this implies that the Galois representations are isomorphic. If $d_1/c^2d_2^3=e^8$, then $(x,y)\mapsto\big(cd_2e^2/x,c^2d_2^2e^5y/x^3\big)$ defines an isomorphism over~$\Q$. Finally, if $d_1/c^2d_2^3=16e^8$, then $(x,y)\mapsto\big(2cd_2e^2/x,4\sqrt2c^2d_2^2e^5y/x^3\big)$ yields an isomorphism over~$\Q\big(\sqrt 2\big)$. We conclude that the $L$-functions are equal under the assumption.
\end{proof}

\begin{Lemma} \label{lemma: N=8 main lemma} Let $p$ be a prime congruent to~$1$ modulo~$8$, $\fP$ a~prime of $\Z\big[\sqrt{-2}\big]$ lying over~$p$, and~$\fp$ a~prime of $\Z[\zeta_8]$ lying over $\fP$. Let $b,c\in\F_p$ with $b\neq 0,\pm1$ and $c\neq 0$. Set $\alpha=\psi_\fp(1+b)^2+\psi_\fp(1-b)^2$. For a character $\eta$ of order $8$ of $\F_p^\times$, let $S_\eta(a)$ be the character sum given in~\eqref{equation: Seta}. Then we have
 \begin{gather*}
 \sum_{j=1,3,5,7}{\psi_\fp(c)^j}S_{\psi_\fp^j}\big(b^2\big)
 =-\big(\psi_\fp(c)\alpha+\psi_\fp(c)^3\overline\alpha\big)\chi(\fP)
 -\big(\psi_\fp(c)^5\alpha+\psi_\fp(c)^7
 \overline\alpha\big)\chi(\overline\fP).
 \end{gather*}
\end{Lemma}

\begin{proof} We consider $b$, $c$ as integers with $p\nmid b,c$. Without loss of generality, we may assume that $c(1+b)^2$ and $c(1-b)^2$ are not integers of the form $\pm e^4$ or $\pm 4e^4$ and that $(1+b)/(1-b)$ is not the square of a rational number.

Let $C_{N,a,c}$ be the curve defined in \eqref{equation: curve C}. We have
 \begin{gather*}
 \# C_{8,b^2,c}(\F_p)=p+1+\sum_{\psi^8=\varepsilon,\, \psi\ne\eps}{\psi(c)}S_\psi\big(b^2\big).
 \end{gather*}
Noticing that $(X,Y)=\big(x,y^2/x\big)$ defines a morphism from $C_{8,b^2,c}$ to $C_{4,b^2,c}$ over $\Q$ and using Lemma~\ref{lemma: morphism}, we find that $C_{8,b^2,c}$ admits morphisms to $E_1\colon y^2=cx(x-1)\big(x-b^2\big)$, $E_2\colon cy^2=x^3+c(1+b)^2x$, $E_3\colon cy^2=x^3+c(1-b)^2x$, $D_{8,c,(1+b)^2}$, and $D_{8,c,(1-b)^2}$. Since $c(1+b)^2$ and $c(1-b)^2$ are not integers of the form $\pm e^4$ or $\pm 4e^4$ and $(1+b)/(1-b)$ is not the square of rational number, by Lemmas~\ref{lemma: same L 4} and~\ref{lemma: same L 8}, the $L$-functions of these five curves are all different. In particular, the Galois representations of~$\operatorname{Gal}\big(\overline\Q/\Q\big)$ associated to these curves are all nonisomorphic. Furthermore, the genus of $C_{8,b^2,c}$ is $7$, which is equal to the sum of the genera of these five curves. Thus, the Galois representations of~$\operatorname{Gal}\big(\overline\Q/\Q\big)$ associated to $D_{8,b^2,c}$ decomposes into the direct sum of the Galois representations associated to these five curves. In other words, we have
 \begin{gather*}
 L(C_{8,b^2,c},s)=L(E_1,s)L(E_2,s)L(E_3,s)L\big(D_{8,(1+b)^2,c},s\big) L\big(D_{8,(1-b)^2,c},s\big).
 \end{gather*}
 On the other hand, we also have
 \begin{gather*}
 \#C_{4,b^2,c}(\F_p)=p+1+\sum_{\psi^4=\varepsilon,\, \psi\neq \eps}\psi(c)S_\psi\big(b^2\big)
 \end{gather*}
 and
 \begin{gather*}
 L(C_{4,b^2,c},s)=L(E_1,s)L(E_2,s)L(E_3,s).
 \end{gather*}
It follows that if we assume that the $p$-factors of $L\big(D_{8,(1+b)^2,c},s\big)$ and $L\big(D_{8,(1-b)^2,c},s\big)$ are $1-up^{-s}+\cdots$ and $1-vp^{-s}+\cdots$, respectively, then
 \begin{gather*}
 \sum_{\psi^8=\varepsilon,\, \psi^4\neq\varepsilon}{\psi(c)}S_\psi\big(b^2\big)=-u-v.
 \end{gather*}
 By Lemma~\ref{lemma: DNc},
 \begin{gather*}
 -u-v=\sum_{\psi^8=\varepsilon,\, \psi^4\neq\varepsilon}
 \big(\psi\big({-}c(1+b)^2\big)+\psi\big({-}c(1-b)^2\big)\big)J(\phi,\psi).
 \end{gather*}
Then from Lemma \ref{lemma: L of D81}, we deduce the formula immediately.
\end{proof}

\begin{proof}[Proof of Theorem \ref{theorem: N=8}] In equation~(\ref{eqn: trace}), we have shown that if $a$ is not a~square in~$\F_p^\times$, then the value of the $_2F_1$-hypergeometric function is~$0$.

Now assume that $a=b^2$ for some $b\in\F_p^\times$. Let $g$ be a~generator of $\F_p^\times$ with $\psi_\fp(g)=\zeta_8$. We apply Lemma~\ref{lemma: N=8 main lemma} with $c=1,g,g^2,g^3$. Setting $\zeta=\zeta_8$, $\alpha=\psi_\fp(1+b)^2+\psi_\fp(1-b)^2$, $\chi=\chi(\fP)$, and $S_j=S_{\psi_\fp^j}\big(b^2\big)$ for $j=1,3,5,7$, we have
\begin{gather*}
 S_1+S_3+S_5+S_7=-(\alpha+\overline\alpha)\chi -(\alpha+\overline\alpha)\overline\chi,\\
 \zeta S_1+\zeta^3 S_3+\zeta^5 S_5+\zeta^7S_7 =-\big(\zeta\alpha+\zeta^3\overline\alpha\big)\chi
 -\big(\zeta^5\alpha+\zeta^7\overline\alpha\big)\overline\chi, \\
 iS_1-iS_3+iS_5-iS_7=-(i\alpha-i\overline\alpha)\chi-(i\alpha-i\overline\alpha)\overline\chi, \\
 \zeta^3 S_1+\zeta S_3+\zeta^7S_5+\zeta^5S_7=-(\zeta^3\alpha+\zeta\overline\alpha)\chi
 -\big(\zeta^7\alpha+\zeta^5\overline\alpha\big)\overline\chi.
 \end{gather*}
 From the relations, we deduce that
 \begin{gather*}
 S_1=-\alpha\chi, \qquad S_3=-\overline\alpha\chi, \qquad
 S_5=-\alpha\overline\chi, \qquad S_7=-\overline\alpha\overline\chi.
 \end{gather*}
 By \eqref{equation: prelim 1} and Lemma~\ref{lemma: reduce}, it follows that
 \begin{align*}
 {}_2F_1\(\begin{matrix}
 \phi\psi_\fp& \psi_\fp\\
 & \phi
 \end{matrix};a\)
 &=-\frac{\phi\psi_\fp(-1)}p\ol\psi_\fp\big(1-b^2\big)^2\alpha\chi \\
 &=-\frac1p(-1)^{(p-1)/8}\big(\ol \psi_\fp(1+b)^2 +\ol\psi_\fp(1-b)^2\big)\chi(\fP).
 \end{align*}

This completes the proof of the theorem.
\end{proof}

\section[Case $N=3, 6, 12$]{Case $\boldsymbol{N=3, 6, 12}$}

Using similar arguments as in the cases of $N=4$ and $8$, we can also obtain the evaluations of hypergeometric functions when the order $N$ of the character is $3$, $6$, or~$12$.

For $N=3$, $6$, we let $\zeta_6=e^{2\pi i/6}$ and~$\sigma$ be the nontrivial element in~$\operatorname{Gal}(\Q(\zeta_6)/\Q)$. For a~prime~$p$ congruent to~$1$ modulo $6$ and a prime $\fp$ of $\Z[\zeta_6]$ lying over $p$, we let $\psi_\fp$ be the character of order $6$ in $\wFx$ satisfying $x^{(p-1)/6}\equiv\psi_\fp(x)\mod\fp$ for all $x\in\Z$. Also, let $E$ be the elliptic curve $y^2=x^3+1$. This elliptic curve has CM by~$\Z[\zeta_6]$. We let $\chi$ be the Hecke character of~$\Z[\zeta_6]$ associated to~$E$ satisfying $\chi(\fP)\in\fP$ for all prime ideals~$\fP$ of~$\Z[\zeta_6]$.

Let $C_{6,a,c}\colon y^6= cx^2(x-1)(x-a)$ be the curve in~\eqref{equation: curve C}, where $a$ and $c$ are integers such that $a=b^2$ is square, $a\not\equiv0,1\mod p$, and $c\not\equiv 0\mod p$. Count the number of points of $C_{6,a,c}$ over $\F_p$ in two ways. Firstly, from the equation, we have
\begin{gather} \label{equation: 3 6 1}
 \# C_{6,a,c}(\F_p)=p+{1+2\phi(c)}+\sum_{k=1}^5\psi_\fp(c)^kS_{\psi_\fp^k}(a),
\end{gather}
where for a character $\psi$ of $\F_p^\times$, $S_\psi$ is defined by
\eqref{equation: Seta}. (The points $\infty$ and $(0,0)$ are singular
points on $C_{6,a,c}$. The resolution of singularity at each of the
two points yields $1+\phi(c)$ points. The term $2\phi(c)$ accounts for
the discrepency.)

On the other hand, similar to the case of $N=8$, if~$b$ is suitably chosen, then
 \begin{gather*}
 L\big( C_{6,a,c}/\Q,s\big)=L\big( C_{3,a,c}/\Q,s\big)L\big( E_{+,b,c}/\Q,s\big)L\big( E_{-,b,c}/\Q,s\big)
 \end{gather*}
where
\begin{gather*}
 C_{3,a,c}\colon \ y^3=cx^2(x-1)(x-a)
\end{gather*}
is a curve of genus $2$ and the elliptic curves
\begin{alignat*}{3}
 & E_{+,b,c}\colon \ && cy^2=x^4+c(1+ b)^2x,& \\
 & E_{-,b,c}\colon \ && cy^2=x^4+c(1- b)^2x &
\end{alignat*}
are given in Lemma~\ref{lemma: morphism}. Assume that the $p$-factor
of $L\big(E_{\pm,b,c}/\Q,s\big)$ is $1-u_\pm p^{-s}+p^{1-2s}$ and that of
$L(C_{3,a,c}/\Q,s)$ is $1-vp^{-s}+\cdots$, so that
\begin{gather*} \label{equation: 3 6 2}
 \# C_{6,a,c}(\F_p)=p+1-u_+-u_--v.
\end{gather*}
Counting $\# C_{3,a,c}(\F_p)$, we obtain
\begin{gather*} \label{equation: 3 6 3}
 p+1+\sum_{k=2,4}\psi_\fp(c)^kS_{\psi_\fp^k}(a)=p+1-v.
\end{gather*}
Combining \eqref{equation: 3 6 1}, \eqref{equation: 3 6 2}, \eqref{equation: 3 6 3}, we find
\begin{gather} \label{equation: 3 6 4}
 \sum_{k=1,3,5}\psi_\fp(c)^kS_{\psi_\fp^k}(a)=-u_+-u_--2\phi(c).
\end{gather}
Now for $k=3$, we have $\psi_\fp^3=\phi$ and
\begin{gather*}
 S_{\psi_\fp^3}(a)=\sum_{x\in\F_p}\phi\big(x^2\big)\phi(x-1)\phi(x-a)
 =\sum_{x\neq 0}\phi(x-1)\phi(x-a).
\end{gather*}
Using the fact that $J(\phi,\phi)=-\phi(-1)$, we can show that the sum above is~$-1-\phi(a)$, which, by assumption that $a=b^2$, is equal to~$-2$. Thus, \eqref{equation: 3 6 4} reduces to
\begin{gather} \label{equation: 3 6 5}
 \psi_\fp(c)S_{\psi_\fp}(a)+\ol\psi_\fp(c)S_{\ol\psi_\fp}(a) =-u_{+}-u_{-}.
\end{gather}
Recall from Lemma \ref{lemma: DNc} that
\begin{gather*}
 p+1-u_\pm=p+1+\phi(c)+\sum_{k=1,3,5}\psi_\fp^k\big({-}c(1\pm b)^2\big)
 J\big(\phi,\psi_\fp^k\big).
\end{gather*}
Again, for $k=3$, we have $\psi_\fp^3=\phi$ and
\begin{gather*}
 \psi_\fp^3\big({-}c(1\pm b)^2\big)J\big(\phi,\psi_\fp^k\big)
 =\phi(-c)J(\phi,\phi)=-\phi(c).
\end{gather*}
It follows that
\begin{gather*}
 u_\pm=-\psi_\fp\big({-}c(1\pm b)^2\big)J(\phi,\psi_\fp)
 -\ol\psi_\fp\big({-}c(1\pm b)^2\big)J\big(\phi,\ol\psi_\fp\big).
\end{gather*}
Plugging this into \eqref{equation: 3 6 5}, we see that
 \begin{gather*}
 \psi_\fp(c)S_{\psi_\fp}(a)+\ol \psi_\fp(c)S_{\ol\psi_\fp}(a)
=\psi_\fp(-c)\(\psi_\fp^2(1+b)+\psi_\fp^2(1-b)\)J(\phi,\psi_\fp)\\
\qquad{}
 +\ol\psi_\fp(-c) \(\ol\psi_\fp^2(1+b)+\ol\psi_\fp^2(1-b)\)J\big(\phi,\ol\psi_\fp\big).
 \end{gather*}
By choosing two suitable $c$, one with $\phi_\fp(c)=1$ and one with $\psi_\fp(c)=e^{2\pi i/3}$, we can determine the value of $S_{\psi_\fp}(a)$. We find that
\begin{align*}
 S_{\psi_\fp}(a)& =\psi_\fp(-1)\(\psi_\fp(1+b)^2+\psi_\fp(1-b)^2\)J(\phi,\psi_\fp)\\
& =-\(\psi_\fp(1+b)^2+\psi_\fp(1-b)^2\)\chi(\fp).
\end{align*}
From the discussion in \cite[Chapters~3 and~6]{Berndt-Evans-Williams}, we see $\psi_\fp(-1)J(\phi,\psi_\fp)=J\big(\phi,\psi_\fp^2\big)=-\chi(\fp)$.
Then by Lemma \ref{lemma: reduce},
\begin{gather*}
 \hgq{\phi\psi_\fp^2}{\psi_\fp^2}{\phi}a
 =\frac{\phi\psi_\fp(-1)^2}pS_{\psi_\fp}(a)
 =-\frac{\phi(-1)}p\big(\psi_\fp(1+b)^2+\psi_\fp(1-b)^2\big)\chi(\fp)
\end{gather*}
and
\begin{align*}
 \hgq{\phi\psi_\fp}{\psi_\fp}\phi a
&=\frac{\phi\psi_\fp(-1)}p\ol\psi_\fp(1-a)^2S_{\psi_\fp}(a) \\
&=-\frac1p\ol\psi_\fp\big((1+b)(1-b)^2\big)
 \big(\psi_\fp(1+b)^2+\psi_\fp(1-b)^2\big)\chi(\fp)
\\
&=-\frac1p\big(\ol\psi_\fp(1+b)^2+\ol\psi_\fp(1-b)^2\big)\chi(\fp).
\end{align*}

We now consider the case $N=12$. Let $\zeta_{12}=e^{2\pi i/12}$ and assume that $p$ is a prime congruent to~$1$ modulo~$12$. Let~$\fp$ be a prime of $\Z[\zeta_{12}]$ lying over $p$, and $\psi_\fp$ be the character of \smash{$(\Z[\zeta_{12}]/\fp)^\times\simeq\F_p^\times$} of order $12$ satisfying $x^{(p-1)/12}\equiv\psi_\fp(x)\mod\fp$ for all $x\in\Z[\zeta_{12}]$.

For a generic integer $c$, and rational numbers $a$, $b$ with $a=b^2$, the $L$-function of the curve
\begin{gather*}
 C_{12,a,c}\colon \ y^{12}= cx^5(x-1)(x-a)
\end{gather*}
is equal to
 \begin{gather*}
 L( C_{12,a,c}/\Q,s)=L( C/\Q,s)L( D_{+,b,c}/\Q,s)L( D_{-,b,c}/\Q,s),
 \end{gather*}
where
 \begin{alignat*}{3}
& C\colon \ && y^6=cx^5(x-1)(x-a),&\\
& D_{+,b,c}\colon \ && cy^2=x^7+c(1+b)^2x,& \\
& D_{-,b,c}\colon && cy^2=x^7+c(1-b)^2x,&
 \end{alignat*}
which are of genus $5$, $3$, and $3$, respectively. Furthermore, the curve $D_{\pm,b,c}$ is a $3$-fold cover of the elliptic curve
\begin{gather*}
 E_{\pm,b,c}\colon \ cY^2=X^3+c(1\pm b)^2X
\end{gather*}
with the covering given by
\begin{gather*}
 D_{\pm,b,c} \longrightarrow E_{\pm,b,c},\qquad (x,y)\mapsto\big(x^3, xy\big).
\end{gather*}
Hence, by Lemma \ref{lemma: DNc}, we have
\begin{align} \label{equation: N=12 temp 2}
 \sum_{j=1,5,7,11}{\psi_\fp(c)^j}S_{j}=& \sum_{j=1,5,7,11}\psi_\fp(-c)^jJ\big(\phi,\psi_\fp^j\big)
 \big(\psi_\fp^{2j}(1+b)+\psi_\fp^{2j}(1-b)\big),
\end{align}
where we let $S_i$ denote $S_{\psi_\fp^i}(a)$. The main task here is to
evaluate $J\big(\phi,\psi_\fp^i\big)$.

For $j=1,5,7,11$, let $\sigma_j$ denote the element of~$\operatorname{Gal}(\Q(\zeta_{12})/\Q)$ such that $\sigma(\zeta_{12})=\zeta_{12}^j$ and set
\begin{gather*}
 \fp_j=\sigma_j(\fp).
\end{gather*}
Let also $\fP$ be the prime of $\Z[i]$ lying below $\fp$ and assume that $a+bi$ is a generating element of~$\fP$. Arguing as in the proof of Lemma~\ref{lemma: L of D81}, we find that
\begin{gather*}
 J(\phi,\psi_\fp)=J\big(\phi,\psi_\fp^5\big)\in\fp_1\fp_5=\fP\Z[\zeta_{12}],
 \qquad J\big(\phi,\psi_\fp^7\big)=J\big(\phi,\psi_\fp^{11}\big)\in\fp_7\fp_{11}
 \ol\fP\Z[\zeta_{12}].
\end{gather*}
Since $|J\big(\phi,\psi_\fp^j\big)|=p$, it follows that
\begin{gather*}
 J(\phi,\psi_\fp)=J\big(\phi,\psi_\fp^5\big)=u(a+bi), \qquad
 J\big(\phi,\psi_\fp^7\big)=J\big(\phi,\psi_\fp^{11}\big)=\ol u(a-bi)
\end{gather*}
for some root of unity $u$ in $\Z[\zeta_{12}]$. As in~\eqref{equation: N=8 temp}, we can show that this root of unity must be one of $\pm 1$ and $\pm i$. In particular,
\begin{gather} \label{equation: N=12 temp}
 J(\phi,\psi_\fp)+J\big(\phi,\psi_\fp^5\big)+J\big(\phi,\psi_\fp^7\big)
 +J\big(\phi,\psi_\fp^{11}\big)=\pm 4a \ \text{or} \ \pm 4b.
\end{gather}

On the other hand, consider the hyperelliptic curve $y^2=x^7+x$. We can easily find two morphisms to the two elliptic curves
\begin{alignat*}{3}
 &E_1\colon \ && Y^2=X^3+X, \quad \mbox{with} \ (X,Y)=\big(x^3,xy\big),& \\
 &E_2 \colon \ && Y^2=X^3-3X, \quad \mbox{with} \ (X,Y)=\(\frac{x^2+1}x,\frac
 y{x^2}\). &
\end{alignat*}
Let $\chi_1$ and $\chi$ be the Hecke characters attached to the elliptic curve $E_1$ and~$E_2$, respectively. Then the reciprocal of the Euler $p$-factor of the $L$-function is
\begin{gather*}
 \big(1-\chi_1(\fP)p^{-s}\big)\big(1-\chi_1\big(\ol\fP\big)p^{-s}\big)
 \big(1-\chi(\fP)p^{-s}\big)\big(1-\chi\big(\ol\fP\big)p^{-s}\big) \big(1-tp^{-s}+p^{1-2s}\big)
\end{gather*}
for some integer $t$. Applying Lemma \ref{lemma:
 DNc} to the curves $y^2=x^7+x$ and $y^2=x^3+x$, we obtain
\begin{gather*}
 \psi_\fp(-1)\sum_{j=1,3,5,7,9,11}J\big(\phi,\psi_\fp^j\big)=
 -\chi_1(\fP)-\chi_1\big(\ol\fP\big)-\chi(\fP)-\chi\big(\ol\fP\big)-t,\\
 \psi_\fp(-1)\big(J\big(\phi,\psi_\fp^3\big)+J\big(\phi,\psi_\fp^9\big)\big)
 =-\chi_1(\fP)-\chi_1\big(\ol\fP\big)
\end{gather*}
and consequently
\begin{gather*}
 \psi_\fp(-1)\sum_{j=1,5,7,11}J\big(\phi,\psi_\fp^j\big)
 =-\chi(\fP)-\chi\big(\ol\fP\big)-t.
\end{gather*}
Now the value of $\chi(\fP)$ is one of $\pm a\pm bi$ and $\pm b\pm
ai$ and by~\eqref{equation: N=12 temp}, $\sum_{j=1,5,7,11}J\big(\phi,\psi_\fp^j\big)$ is one of~$\pm 4a$ and~$\pm 4b$.
Therefore, we must have $t=\chi(\fP)+\chi\big(\ol\fP\big)$. It follows that
\begin{gather*}
 \psi_\fp(-1)J(\phi,\psi_\fp)+\chi(\fP)
 =-\psi_\fp(-1)J\big(\phi,\psi_\fp^{11}\big)-\chi\big(\ol\fP\big),
\end{gather*}
which lies in $\fP\ol\fP=p\Z[i]$. Since the absolute values of
$J\big(\phi,\psi_\fp^j\big)$ and $\chi(\fP)$ are all $\sqrt p$, the only
possibility is that
\begin{gather*}
 \psi_\fp(-1)J(\phi,\psi_\fp)=-\chi(\fP), \qquad
 \psi_\fp(-1)J\big(\phi,\psi_\fp^{11}\big)=-\chi\big(\ol\fP\big).
\end{gather*}

Now choose an integer $g$ such that $\psi_\fp(g)=\zeta_{12}$. Set $\zeta=\zeta_{12}$, $\alpha=\psi_\fp(1+b)^2+\psi_\fp(1-b)^2$, and $\chi=\chi(\fP)$. Applying~\eqref{equation: N=12 temp 2} with $c=1,g,g^2,g^3$, we obtain
\begin{gather*}
 S_1+S_5+S_7+S_{11}=-(\alpha+\overline\alpha)\chi -(\alpha+\ol\alpha)\ol\chi, \\
 \zeta S_1+\zeta^5 S_5+\zeta^7S_7+\zeta^{11}S_{11}=-\big(\zeta\alpha+\zeta^5\ol\alpha\big)\chi
 -\big(\zeta^7\alpha+\zeta^{11}\ol\alpha\big)\ol\chi, \\
 \zeta^{2}S_1+\zeta^{10}S_5+\zeta^{2}S_7+\zeta^{10}S_{11}
 =-\big(\zeta^2\alpha+\zeta^{10}\ol\alpha\big)\chi -\big(\zeta^2\alpha+\zeta^{10}\ol\alpha\big)\ol\chi, \\
 iS_1+iS_5-iS_7-iS_{11}=-(i\alpha+i\ol\alpha)\chi +(i\alpha+i\ol\alpha)\ol\chi.
\end{gather*}
From these identities, we deduce that
 \begin{gather*}
 S_1=- \alpha\chi, \qquad S_5=-\ol \alpha\chi, \qquad
 S_7=- \alpha\overline\chi, \qquad S_{11}=-\ol\alpha\overline\chi.
 \end{gather*}
Using Lemma \ref{lemma: reduce}, we finally arrive at
\begin{align*}
 \hgq{\phi\psi_\fp}{\psi_\fp}\phi a
&=\frac{\phi\psi_\fp(-1)}p\ol\psi_\fp(1-a)^2 S_{\psi_\fp}(a) \\
&=-\frac{(-1)^{(p-1)/12}}p\ol\psi_\fp(1-a)^2
 \big(\psi_\fp(1+b)^2+\psi_\fp(1-b)^2\big)\chi(\fP) \\
&=-\frac{(-1)^{(p-1)/12}}p
 \big(\ol\psi_\fp(1+b)^2+\ol\psi_\fp(1-b)^2\big)\chi(\fP).
\end{align*}
This proves Theorem~\ref{theorem: N=12}.

 \subsection*{Acknowledgements}
 The authors would like to thank the National Center for Theoretical Sciences (NCTS) in Taiwan. Tu was supported in part by the NCTS where this project was initiated. Yang was partially supported by Grant 102-2115-M-009-001-MY4 of the Ministry of Science and Technology, Taiwan (R.O.C.). We also thank the Banff International Research Station in Canada for the workshop on Modular Forms on String Theory and the International Mathematical Research Institute MATRIX in Australia for the workshop on Hypergeometric Motives and Calabi--Yau Differential Equations, where we further collaborated on this work. We are very grateful to Ling Long for her interest and enlightening discussions. The authors would like to thank the referees who provided many helpful comments.

\pdfbookmark[1]{References}{ref}
\LastPageEnding


\begin{thebibliography}{99}
\footnotesize\itemsep=0pt

\bibitem{Ahlgren-Ono-hypergeometric}
Ahlgren S., Ono K., A {G}aussian hypergeometric series evaluation and {A}p\'ery
 number congruences, \href{https://doi.org/10.1515/crll.2000.004}{\textit{J.~Reine Angew. Math.}} \textbf{518} (2000),
 187--212.

\bibitem{Ahlgren-Ono-CalabiYau}
Ahlgren S., Ono K., Modularity of a certain {C}alabi--{Y}au threefold,
 \href{https://doi.org/10.1007/s006050050069}{\textit{Monatsh. Math.}} \textbf{129} (2000), 177--190.

\bibitem{AOP}
Ahlgren S., Ono K., Penniston D., Zeta functions of an infinite family of
 {$K3$} surfaces, \href{https://doi.org/10.1353/ajm.2002.0007}{\textit{Amer.~J. Math.}} \textbf{124} (2002), 353--368.

\bibitem{BK-certain}
Barman R., Kalita G., Certain values of {G}aussian hypergeometric series and a
 family of algebraic curves, \href{https://doi.org/10.1142/S179304211250056X}{\textit{Int.~J. Number Theory}} \textbf{8} (2012),
 945--961, \href{https://arxiv.org/abs/1208.0495}{arXiv:1208.0495}.

\bibitem{BK-elliptic}
Barman R., Kalita G., Elliptic curves and special values of {G}aussian
 hypergeometric series, \href{https://doi.org/10.1016/j.jnt.2013.03.010}{\textit{J.~Number Theory}} \textbf{133} (2013), 3099--3111.

\bibitem{Berndt-Evans-Williams}
Berndt B.C., Evans R.J., Williams K.S., Gauss and {J}acobi sums, \textit{Canadian
 Mathematical Society Series of Monographs and Advanced Texts}, John Wiley \&
 Sons, Inc., New York, 1998.

\bibitem{Evans-Greene-Clausentheorem}
Evans R., Greene J., Clausen's theorem and hypergeometric functions over finite
 fields, \href{https://doi.org/10.1016/j.ffa.2008.09.001}{\textit{Finite Fields Appl.}} \textbf{15} (2009), 97--109.

\bibitem{Evans-Greene-Evaluations}
Evans R., Greene J., Evaluations of hypergeometric functions over finite
 fields, \textit{Hiroshima Math.~J.} \textbf{39} (2009), 217--235.

\bibitem{Evans-Lam}
Evans R., Lam F., Special values of hypergeometric functions over finite
 fields, \href{https://doi.org/10.1007/s11139-007-9116-7}{\textit{Ramanujan~J.}} \textbf{19} (2009), 151--162.

\bibitem{Frechette-Ono-Papanikolas}
Frechette S., Ono K., Papanikolas M., Gaussian hypergeometric functions and
 traces of {H}ecke operators, \href{https://doi.org/10.1155/S1073792804132522}{\textit{Int. Math. Res. Not.}} \textbf{2004}
 (2004), 3233--3262.

\bibitem{Fuselier-PhD}
Fuselier J.G., Hypergeometric functions over finite fields and relations to
 modular forms and elliptic curves, Ph.D.~Thesis, Texas A\&M University, 2007.

\bibitem{Fuselier}
Fuselier J.G., Hypergeometric functions over {${\mathbb F}_p$} and relations to
 elliptic curves and modular forms, \href{https://doi.org/10.1090/S0002-9939-09-10068-0}{\textit{Proc. Amer. Math. Soc.}}
 \textbf{138} (2010), 109--123, \href{https://arxiv.org/abs/0805.2885}{arXiv:0805.2885}.

\bibitem{Fuselier-level1}
Fuselier J.G., Traces of {H}ecke operators in level~1 and {G}aussian
 hypergeometric functions, \href{https://doi.org/10.1090/S0002-9939-2012-11540-0}{\textit{Proc. Amer. Math. Soc.}} \textbf{141}
 (2013), 1871--1881, \href{https://arxiv.org/abs/1109.3362}{arXiv:1109.3362}.

\bibitem{win3c}
Fuselier J.G., Long L., Ramakrishna R., Swisher H., Tu F.-T., Hypergeometric
 functions over finite fields, \href{https://arxiv.org/abs/1510.02575}{arXiv:1510.02575}.

\bibitem{Goodson}
Goodson H., Hypergeometric properties of genus 3 generalized {L}egendre curves,
 \href{https://doi.org/10.1016/j.jnt.2017.09.019}{\textit{J.~Number Theory}} \textbf{186} (2018), 121--136, \href{https://arxiv.org/abs/1705.02404}{arXiv:1705.02404}.

\bibitem{Greene}
Greene J., Hypergeometric functions over finite fields, \href{https://doi.org/10.2307/2000329}{\textit{Trans. Amer.
 Math. Soc.}} \textbf{301} (1987), 77--101.

\bibitem{Greene-Stanton}
Greene J., Stanton D., A character sum evaluation and {G}aussian hypergeometric
 series, \href{https://doi.org/10.1016/0022-314X(86)90009-0}{\textit{J.~Number Theory}} \textbf{23} (1986), 136--148.

\bibitem{Hashimoto-Long-Yang}
Hashimoto K.-I., Long L., Yang Y., Jacobsthal identity for {${\mathbb
 Q}\big(\sqrt{-2}\big)$}, \href{https://doi.org/10.1515/form.2011.102}{\textit{Forum Math.}} \textbf{24} (2012), 1125--1138,
 \href{https://arxiv.org/abs/1110.5815}{arXiv:1110.5815}.

\bibitem{Ireland-Rosen}
Ireland K., Rosen M., A classical introduction to modern number theory,
 \href{https://doi.org/10.1007/978-1-4757-2103-4}{\textit{Graduate Texts in Mathematics}}, Vol.~84, 2nd~ed., Springer-Verlag,
 New York, 1990.

\bibitem{Koike-Aperynumbers}
Koike M., Hypergeometric series over finite fields and {A}p\'ery numbers,
 \textit{Hiroshima Math.~J.} \textbf{22} (1992), 461--467.

\bibitem{Koike-ellipticcurves}
Koike M., Hypergeometric series and elliptic curves over finite fields,
 \textit{S\=urikaisekikenky\=usho K\=oky\=uroku} \textbf{843} (1993), 27--35.

\bibitem{Lennon-elliptic}
Lennon C., Gaussian hypergeometric evaluations of traces of {F}robenius for
 elliptic curves, \href{https://doi.org/10.1090/S0002-9939-2010-10609-3}{\textit{Proc. Amer. Math. Soc.}} \textbf{139} (2011),
 1931--1938, \href{https://arxiv.org/abs/1003.4421}{arXiv:1003.4421}.

\bibitem{Lennon-modularity}
Lennon C., Trace formulas for {H}ecke operators, {G}aussian hypergeometric
 functions, and the modularity of a threefold, \href{https://doi.org/10.1016/j.jnt.2011.05.005}{\textit{J.~Number Theory}}
 \textbf{131} (2011), 2320--2351, \href{https://arxiv.org/abs/1003.1157}{arXiv:1003.1157}.

\bibitem{McCarthy-Papanikolas}
McCarthy D., Papanikolas M.A., A finite field hypergeometric function
 associated to eigenvalues of a {S}iegel eigenform, \href{https://doi.org/10.1142/S1793042115501134}{\textit{Int.~J. Number
 Theory}} \textbf{11} (2015), 2431--2450, \href{https://arxiv.org/abs/1205.1006}{arXiv:1205.1006}.

\bibitem{Ono}
Ono K., Values of {G}aussian hypergeometric series, \href{https://doi.org/10.1090/S0002-9947-98-01887-X}{\textit{Trans. Amer. Math.
 Soc.}} \textbf{350} (1998), 1205--1223.

\bibitem{Salerno}
Salerno A., Counting points over finite fields and hypergeometric functions,
 \href{https://doi.org/10.7169/facm/2013.49.1.9}{\textit{Funct. Approx. Comment. Math.}} \textbf{49} (2013), 137--157,
 \href{https://arxiv.org/abs/1201.3335}{arXiv:1201.3335}.

\bibitem{Silverman}
Silverman J.H., The arithmetic of elliptic curves, \href{https://doi.org/10.1007/978-1-4757-1920-8}{\textit{Graduate Texts in
 Mathematics}}, Vol.~106, Springer-Verlag, New York, 1986.

\bibitem{Silverman-advence}
Silverman J.H., Advanced topics in the arithmetic of elliptic curves,
 \href{https://doi.org/10.1007/978-1-4612-0851-8}{\textit{Graduate Texts in Mathematics}}, Vol.~151, Springer-Verlag, New York,
 1994.

\bibitem{Weil}
Weil A., Jacobi sums as ``{G}r\"ossencharaktere'', \href{https://doi.org/10.2307/1990804}{\textit{Trans. Amer. Math.
 Soc.}} \textbf{73} (1952), 487--495.

\end{thebibliography}
\end{document}